\documentclass[11pt]{article}
\usepackage{amssymb,amsmath,amsfonts,amsthm,rotating,setspace,mathtools,indentfirst,graphics,color,float,graphicx}

\usepackage{epsfig}

\numberwithin{equation}{section}

\theoremstyle{plain}
\newtheorem{thm}{Theorem}[section]
\newtheorem{lem}{Lemma}[section]

\theoremstyle{definition}

\newtheorem{rem}{Remark}

\begin{document}

\title{\vskip-0.3in Pointwise Bounds and Blow-up for Choquard-Pekar Inequalities at an Isolated Singularity}

\author{Marius Ghergu\footnote{School of Mathematics and Statistics,
    University College Dublin, Belfield, Dublin 4, Ireland; {\tt
      marius.ghergu@ucd.ie}} $\;\;$ and $\;$
Steven D. Taliaferro\footnote{Mathematics Department, Texas A\&M
    University, College Station, TX 77843-3368; {\tt stalia@math.tamu.edu}} 
\footnote{Corresponding author, Phone 001-979-845-3261, Fax
  001-979-845-6028}}

\date{}
\maketitle

\begin{abstract}
We study the behavior near the origin in $\mathbb{R}^n ,n\geq3$, of nonnegative functions
\begin{equation}\label{0.1}
 u\in C^2 (\mathbb{R}^n \backslash \{0\})\cap L^\lambda (\mathbb{R}^n )
\end{equation}
satisfying the Choquard-Pekar type inequalities 
\begin{equation}\label{0.2}
 0\leq-\Delta u\leq(|x|^{-\alpha}*u^\lambda )u^\sigma \quad\text{ in }B_2 (0)\backslash \{0\}
\end{equation}
where $\alpha\in(0,n),\lambda>0,$ and $\sigma\geq0$ are constants and
$*$ is the convolution operation in $\mathbb{R}^n$.  We provide
optimal conditions on $\alpha,\lambda$, and $\sigma$ such that
nonnegative solutions $u$ of (\ref{0.1},\ref{0.2}) satisfy
pointwise bounds near the origin.  
\medskip

\noindent {\it Keywords}: pointwise bound; blow-up; isolated singularity;
Choquard-Pekar equation; Riesz potential.
\end{abstract}


\section{Introduction}\label{sec1}
In this paper we study the behavior near the origin in $\mathbb{R}^n ,n\geq3$, of nonnegative functions 
\begin{equation}\label{1.1}
 u\in C^2 (\mathbb{R}^n \backslash \{0\})\cap L^\lambda (\mathbb{R}^n )
\end{equation}
satisfying the Choquard-Pekar type inequalities
\begin{equation}\label{1.2}
 0\leq-\Delta u\leq(|x|^{-\alpha}*u^\lambda )u^\sigma \quad\text{ in }B_2 (0)\backslash \{0\}
\end{equation}
where $\alpha\in(0,n),\lambda>0$, and $\sigma\geq0$ are constants and
$*$ is the convolution operation in $\mathbb{R}^n$. The regularity
condition $u\in L^\lambda(\mathbb{R}^n)$ in \eqref{1.1} is required
for the nonlocal convolution operation in \eqref{1.2} to make sense.

A motivation for the study of (\ref{1.1},\ref{1.2}) comes from the
equation
\begin{equation}\label{prototype}
-\Delta u=(|x|^{-\alpha}* u^{\lambda}) |u|^{\lambda-2}u\quad
\text{ in }\mathbb{R}^n,
\end{equation}
where $\alpha\in (0,n)$ and $\lambda>1$.  For $n=3$, $\alpha=1$,
and $\lambda=2$, equation \eqref{prototype} is known in the literature as the
{\it Choquard-Pekar equation} and was introduced in \cite{P1954} as a
model in quantum theory of a Polaron at rest (see also
\cite{DA2010}). Later, the equation \eqref{prototype} appears as a model
of an electron trapped in its own hole, in an approximation to
Hartree-Fock theory of one-component plasma \cite{L1976}.  More
recently, the same equation \eqref{prototype} was used in a model of
self-gravitating matter (see, e.g., \cite{J1995,MPT1998}) and it is
known in this context as the {\it Schr\"odinger-Newton equation}.

The Choquard-Pekar equation \eqref{prototype} has been investigated
for a few decades by variational methods starting with the pioneering
works of Lieb \cite{L1976} and Lions \cite{Lions1980,Lions1984}. More
recently, new and improved techniques have been devised to deal with
various forms of \eqref{prototype} (see, e.g.,
\cite{MZ2010,MZ2012,MV2013a,MV2013b,MV2015,WW2009} and the references
therein).

Using nonvariational methods, the authors in \cite{MV2013b} obtained sharp conditions for the nonexistence of nonnegative solutions to 
$$-\Delta u \geq (|x|^{-\alpha}* u^{\lambda}) u^{\sigma}
$$
in an exterior domain of $\mathbb{R}^n$, $n\geq 3$.

In this paper we address the following question.

\medskip
\noindent \textbf{Question 1.} Suppose $\alpha\in(0,n)$ and
$\lambda>0$ are constants.  For which nonnegative constants $\sigma$,
if any, does there exist a continuous function
$\varphi:(0,1)\to(0,\infty)$ such that all nonnegative solutions $u$
of (\ref{1.1},\ref{1.2}) satisfy
\begin{equation}\label{1.3}
 u(x)=O(\varphi(|x|))\quad\text{ as }x\to0
\end{equation}
and what is the optimal such $\varphi$ when it exists?

We call the function $\varphi$ in \eqref{1.3} a pointwise bound for $u$ as $x\to0$.

\begin{rem}\label{rem1}
 Let $u_\lambda \in C^2 (\mathbb{R}^n \backslash \{0\})$ be a nonnegative function such that $u_\lambda =0$ in $\mathbb{R}^n \backslash B_3 (0)$ and
  $$u_\lambda (x)= 
 \begin{cases}
  |x|^{-(n-2)} & \text{if }0<\lambda<\frac{n}{n-2}\\
  1& \text{if }\lambda\geq\frac{n}{n-2}    
 \end{cases}
 \quad\text{ for }0<|x|<2.$$ 
Then $u_\lambda \in L^\lambda
 (\mathbb{R}^n )$ and $-\Delta u_\lambda =0$ in $B_2 (0)\backslash
 \{0\}$.  Hence $u_\lambda$ is a solution of (\ref{1.1},\ref{1.2}) 
for all $\alpha\in (0,n)$, $\lambda>0$, and $\sigma\ge 0$. Thus any
pointwise bound for nonnegative solutions $u$ of (\ref{1.1},\ref{1.2})
as $x\to0$ must be at least as large as $u_\lambda (x)$ and whenever
 $u_\lambda (x)$ is such a bound it is necessarily optimal.  In this
 case we say $u$ is harmonically bounded at $0$.
\end{rem}

In order to state our results for Question 1, we define for each
$\alpha\in(0,n)$ the continuous, piecewise linear function $g_\alpha
:(0,\infty)\to [0,\infty)$ by
\begin{equation}\label{1.4}
 g_\alpha (\lambda)=
 \begin{cases}
  \frac{n}{n-2} & \text{if }0<\lambda<\frac{n-\alpha}{n-2}\\
  \frac{2n-\alpha}{n-2}-\lambda & \text{if }\frac{n-\alpha}{n-2}\leq\lambda<\frac{n}{n-2}\\
  \max\{0,1-\frac{\alpha-2}{n}\lambda \} & \text{if }\lambda\geq\frac{n}{n-2}.
 \end{cases}
\end{equation}

According to the following theorem, if the point $(\lambda,\sigma)$
lies below the graph of $\sigma=g_\alpha(\lambda)$ then all
nonnegative solutions $u$ of (\ref{1.1},\ref{1.2}) are harmonically
bounded at $0$.

\begin{thm}\label{thm1.1}
 Suppose $u$ is a nonnegative solution of (\ref{1.1},\ref{1.2}) where $\alpha\in(0,n),\lambda>0$, and 
 $$0\leq\sigma<g_\alpha (\lambda).$$  
 Then $u$ is harmonically bounded at $0$, that is, as $x\to0$,
 \begin{equation}\label{1.5}
  u(x)=
  \begin{cases}
  O(|x|^{-(n-2)}) & \text{if }0<\lambda<\frac{n}{n-2}\\
  O(1) & \text{if }\lambda\geq\frac{n}{n-2}.
  \end{cases}
 \end{equation}
Moreover, if $\lambda\ge \frac{n}{n-2}$ then $u$ has a $C^1$ extension
to the origin, that is, $u=w|_{\mathbb{R}^n\setminus\{0\}}$ for some
function $w\in C^1(\mathbb{R}^n)$.
\end{thm}

By Remark \ref{rem1} the bound \eqref{1.5} for $u$ is optimal.

By the next theorem, if the point $(\lambda,\sigma)$ lies above the
graph of $\sigma=g_\alpha(\lambda)$ then there does not exist a
pointwise bound for nonnegative solutions of (\ref{1.1},\ref{1.2}) as
$x\to0$.

\begin{thm}\label{thm1.2}
 Suppose $\alpha,\lambda$, and $\sigma$ are constants satisfying
 $$\alpha\in(0,n), \quad \lambda>0, \quad \text{and}\quad\ \sigma>g_\alpha (\lambda).$$
 Let $\varphi:(0,1)\to(0,\infty)$ be a continuous function satisfying
 $$\lim_{t\to0^+}\varphi(t)=\infty.$$
 Then there exists a nonnegative solution $u$ of (\ref{1.1},\ref{1.2}) such that
 $$u(x)\neq O(\varphi(|x|))\quad\text{ as }x\to0.$$
\end{thm}

Theorems \ref{thm1.1} and \ref{thm1.2} completely answer Question 1
when the point $(\lambda,\sigma)$ does not lie on the graph of
$g_\alpha$.  Concerning the critical case that $(\lambda,\sigma)$ lies
on the graph of $g_\alpha$ we have the following result.

\begin{thm}\label{thm1.3}
 Suppose $\alpha\in(0,n)$.
 \begin{enumerate}
 \item[(i)] If $0<\lambda<\frac{n-\alpha}{n-2}$ and $\sigma=g_\alpha
   (\lambda)$ then all nonnegative solutions $u$ of
   (\ref{1.1},\ref{1.2}) are harmonically bounded at $0$.
  \item[(ii)] If $\lambda=\frac{n-\alpha}{n-2}$ and $\sigma=g_\alpha
    (\lambda)$ then there does not exist a pointwise bound for
    nonnegative solutions $u$ of (\ref{1.1},\ref{1.2}) as $x\to0$.
  \item[(iii)] If $\alpha\in(2,n)$, $\lambda>\frac{n}{\alpha-2}$, and
    $\sigma=g_\alpha(\lambda)$ then there does not exist a pointwise
    bound for nonnegative solutions $u$ of (\ref{1.1},\ref{1.2}) as
    $x\to0$.
 \end{enumerate}
\end{thm}

If $u$ is a nonnegative solution of (\ref{1.1},\ref{1.2}) where
$(\lambda,\sigma)$ lies in the first quadrant of the
$\lambda\sigma$-plane and $\sigma\not= g_\alpha(\lambda)$ then
according to Theorems \ref{thm1.1} and \ref{thm1.2} either
\begin{enumerate}
\item[(i)] $u$ is bounded around the origin and can be extended
  to a $C^1$ function in the whole $\mathbb{R}^n$; or
\item[(ii)] $u$ can be unbounded around the origin but must satisfy
  $u=O(|x|^{-(n-2)})$ as $x\to 0$; or
\item[(iii)] no pointwise a priori bound  exists for $u$ as $x\to 0$,
  that is solutions can be arbitrarily large around the origin. 
\end{enumerate}
The regions in which these three possibilities occur are depicted in
Figs.~1--3 below.

\begin{figure}[h!]
 \centering
  \includegraphics[width=0.7\textwidth]{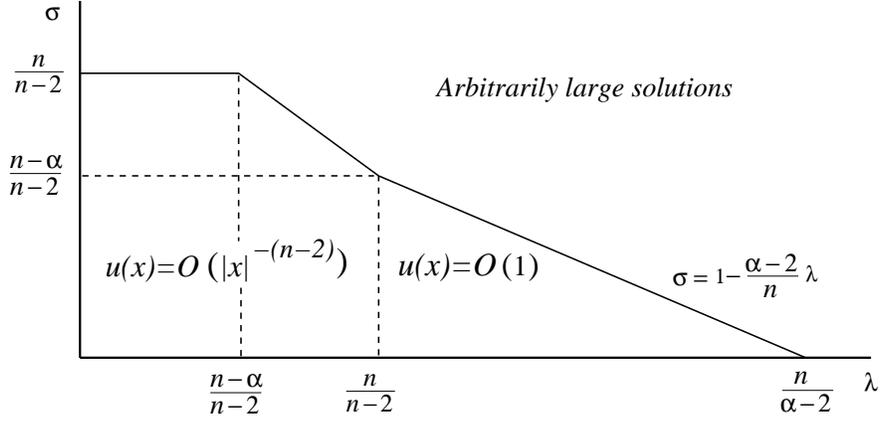}
 \caption{Case $\alpha\in (2,n)$.}
 \end{figure}

\bigskip

\begin{figure}[h!]
 \centering
  \includegraphics[width=0.7\textwidth]{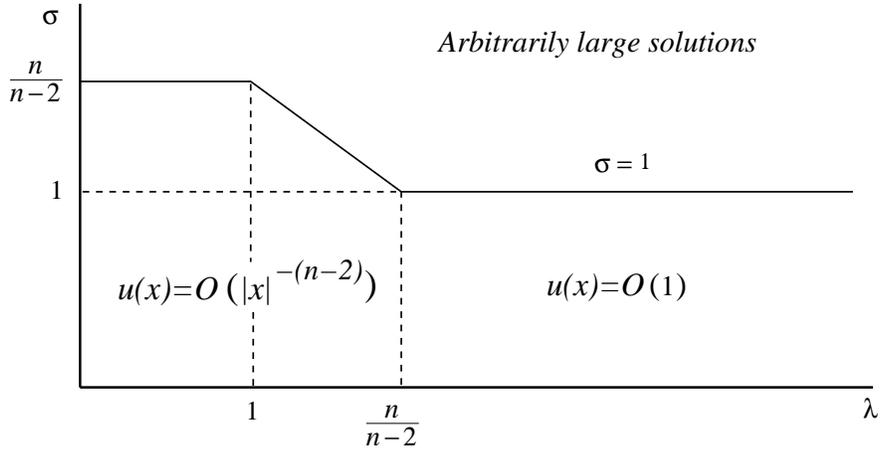}
 \caption{Case $\alpha=2$.}
 \end{figure}

\bigskip

\begin{figure}[h!]
 \centering
  \includegraphics[width=0.7\textwidth]{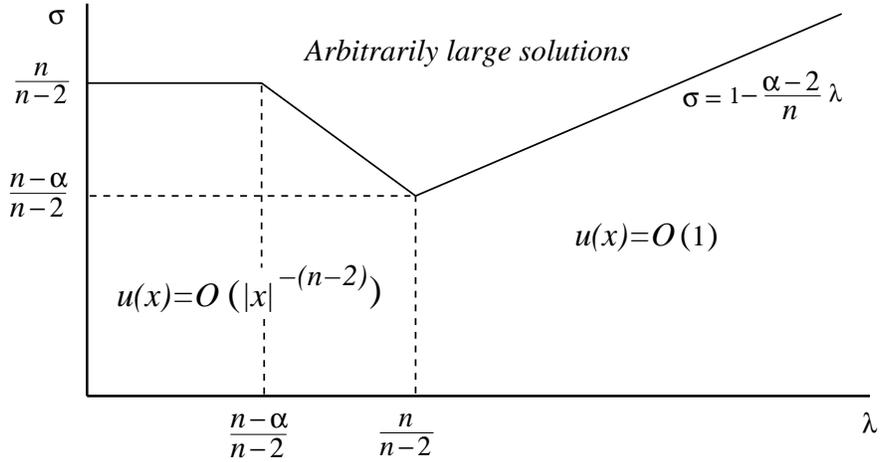}
 \caption{Case $\alpha\in (0,2)$.}
 \end{figure}

 If $\alpha\in(0,n)$ and $\lambda>0$ then one of the following three
 conditions holds:
\begin{enumerate}
 \item[(i)] $0<\lambda<\frac{n-\alpha}{n-2};$
 \item[(ii)] $\frac{n-\alpha}{n-2}\leq\lambda<\frac{n}{n-2};$
 \item[(iii)] $\frac{n}{n-2}\leq\lambda<\infty$.
\end{enumerate}

The proofs of Theorems \ref{thm1.1}--\ref{thm1.3} in case (i)(resp.
(ii), (iii)) are given in Section \ref{sec3} (resp. \ref{sec4},
\ref{sec5}).  In Section \ref{sec2} we provide some lemmas needed for
these proofs. Our approach relies on an integral representation
formula for nonnegative superharmonic functions due to Brezis and Lions
\cite{BL} (see Lemma \ref{lem2.1} below) together with various
integral estimates for Riesz potentials.

Finally we mention that throughout this paper $\omega_n$ denotes the
volume of the unit ball in $\mathbb{R}^n$ and by {\it Riesz potential
  estimates} we mean the estimates given in \cite[Lemma 7.12]{GT1983}
and \cite[Chapter 5, Theorem 1]{Stein}. See also \cite[Appendix C]{GT2016}.

\section{Preliminary lemmas}\label{sec2}
In this section we provide some lemmas needed for the proofs of our results in Sections \ref{sec3}--\ref{sec5}.

\begin{lem}\label{lem2.1}
 Suppose $u$ is a nonnegative solution of (\ref{1.1},\ref{1.2}) for some constants $\alpha\in(0,n),\lambda>0$, and $\sigma\geq0$.  Let $v=u+1$.  Then
 \begin{equation}\label{2.1}
  v\in C^2 (\mathbb{R}^n \backslash \{0\})\cap L^\lambda (B_2 (0))
 \end{equation}
 and, for some positive constant $C$, $v$ satisfies
 \begin{equation}\label{2.2}
  \begin{rcases}
   0\leq-\Delta v\leq C[I_{n-\alpha}(v^\lambda )]v^\sigma \\
   v\geq1
  \end{rcases}
  \text{ in }B_2 (0)\backslash \{0\},
 \end{equation}
 where
\begin{equation}\label{2.2.5} 
(I_\beta f)(x):=\int_{|y|<1}\frac{f(y)dy}{|x-y|^{n-\beta}}\quad\text{for }
 \beta\in (0,n).
\end{equation}
Also
 \begin{equation}\label{2.3}
  -\Delta v,v^\mu \in L^1 (B_1(0))\quad\text{for all }\mu\in[1,\frac{n}{n-2})
 \end{equation}
 and
 \begin{equation}\label{2.4}
  v(x)=\frac{m}{|x|^{n-2}}+h(x)+C\int_{|y|<1}\frac{-\Delta
    v(y)dy}{|x-y|^{n-2}}\quad\text{ for }0<|x|<1
 \end{equation}
 where $m\geq0$ and $C>0$ are constants and $h$ is harmonic and bounded in $B_1 (0)$.
\end{lem}

\begin{proof}
 \eqref{2.1} follows from \eqref{1.1} and the definition of $v$.  
 
 For $0<|x|<2$ we have 
 \begin{align*}
  \int_{|y|>1}\frac{u(y)^\lambda
    dy}{|x-y|^\alpha}&\leq\left(\max_{1\leq |y|\leq 3}u(y)^\lambda \right)\int_{1<|y|<3}\frac{dy}{|x-y|^\alpha}+\int_{|y|>3}u(y)^\lambda dy\\
  &\le C\le C\min_{|z|\leq2}\int_{|y|<1}\frac{dy}{|z-y|^\alpha},
 \end{align*}
 where, as usual, $C$ is a positive constant whose value may change
 from line to line. Thus for $0<|x|<2$
 \begin{align*}
  \int_{\mathbb{R}^n}\frac{u(y)^\lambda dy}{|x-y|^\alpha}&\leq\int_{|y|<1}\frac{u(y)^\lambda dy}{|x-y|^\alpha}+C\int_{|y|<1}\frac{1^\lambda}{|x-y|^\alpha}dy\\
  &\leq C\int_{|y|<1}\frac{(u(y)+1)^\lambda}{|x-y|^\alpha}dy=C[I_{n-\alpha}(v^\lambda )](x).
 \end{align*}
 Hence, since $\Delta u=\Delta v$ and $u<v$ we see that \eqref{2.2} follows from \eqref{1.2}.  Also \eqref{2.1}, \eqref{2.2}, and \cite{BL} imply \eqref{2.3} with $\mu=1$ and \eqref{2.4}, which together with Riesz potential estimates, 
 yield the complete statement \eqref{2.3}.
\end{proof}

The following lemma will be needed for the proof of Theorem \ref{thm1.2} when $0<\lambda\leq\frac{n}{n-2}$.

\begin{lem}\label{lem2.2}
 Suppose $\alpha\in(0,n)$ and $\lambda\in(0,\frac{n}{n-2}]$.  Let $\{x_j \}\subset\mathbb{R}^n ,n\geq3$, and $\{r_j \},\{\varepsilon_j \}\subset(0,1)$ be sequences satisfying
 \begin{equation}\label{2.5}
  0<4|x_{j+1}|<|x_j |<1/2,
 \end{equation}
 \begin{equation}\label{2.6}
  0<r_j <|x_j |/4\quad\text{ and }
\quad\sum^{\infty}_{j=1}(\varepsilon^{\lambda}_{j}+\varepsilon_j) <\infty.
 \end{equation}
 Then there exists a nonnegative function
 \begin{equation}\label{2.7}
  u\in C^\infty (\mathbb{R}^n \backslash \{0\})\cap L^\lambda (\mathbb{R}^n )
 \end{equation}
 such that
 \begin{equation}\label{2.8}
  0\leq-\Delta u\leq
  \begin{Bmatrix}
   \frac{\varepsilon_j}{r^{n}_{j}} \qquad\qquad\quad & \text{if }0<\lambda<\frac{n}{n-2}\\
   \frac{\varepsilon_j}{r^{n}_{j}(\log\frac{1}{r_j})^{\frac{n-2}{n}}} & \text{if }\lambda=\frac{n}{n-2}\qquad
  \end{Bmatrix}
 \quad\text{ in }B_{r_j}(x_j ),
 \end{equation}
 \begin{equation}\label{2.9}
  -\Delta u=0\quad\text{ in }B_2 (0)\backslash(\{0\}\cup\bigcup^{\infty}_{j=1}B_{r_j}(x_j )),
 \end{equation}
 \begin{equation}\label{2.10}
  u\geq A
  \begin{Bmatrix}
   \frac{\varepsilon_j}{r^{n-2}_{j}} \qquad\qquad\quad & \text{if }0<\lambda<\frac{n}{n-2}\\
   \frac{\varepsilon_j}{r^{n-2}_{j}(\log\frac{1}{r_j})^{\frac{n-2}{n}}} & \text{if }\lambda=\frac{n}{n-2}\qquad
  \end{Bmatrix}
  \quad\text{ in }B_{r_j}(x_j),
 \end{equation}
 and for $x\in B_{r_j}(x_j )$
 \begin{equation}\label{2.11}
  \int_{|y|<1}\frac{u(y)^\lambda dy}{|x-y|^\alpha}\geq B\
  \begin{cases}
   \varepsilon^{\lambda}_{j} & \text{if }0<\lambda<\frac{n-\alpha}{n-2}\\
   \varepsilon^{\lambda}_{j}\log\frac{1}{r_j} & \text{if }\lambda=\frac{n-\alpha}{n-2}\\
   \varepsilon^{\lambda}_{j}r^{n-\alpha-(n-2)\lambda}_{j} & \text{if }\frac{n-\alpha}{n-2}<\lambda<\frac{n}{n-2}\\
   \varepsilon^{\lambda}_{j}r^{-\alpha}_{j}(\log\frac{1}{r_j})^{-1} & \text{if }\lambda=\frac{n}{n-2}
  \end{cases}
 \end{equation}
 where $A=A(n)$ and $B=B(n,\lambda,\alpha)$ are positive constants.
\end{lem}

\begin{proof}
 Let $\psi:\mathbb{R}^n \to[0,1]$ be a $C^\infty$ function whose support is $\overline{B_1(0)}$.  Define $\psi_j ,f_j :\mathbb{R}^n \to[0,\infty)$ by $\psi_j (y)=\psi(\eta)$ where $y=x_j +r_j \eta$ and $f_j =M_j \psi_j$ where 
 $M_j=\frac{\varepsilon_j}{r^{n}_{j}\delta_j}$ and 
 $$\delta_j =
 \begin{cases}
  1 & \text{if }0<\lambda<\frac{n}{n-2}\\
  (\log\frac{1}{r_j})^{\frac{n-2}{n}} & \text{if }\lambda=\frac{n}{n-2}.
 \end{cases}
 $$
 Since
 $$\int_{\mathbb{R}^n}f_j (y)dy=M_j \int_{\mathbb{R}^n}\psi(\eta)r^{n}_{j}d\eta\leq\varepsilon_j \int_{\mathbb{R}^n}\psi(\eta)d\eta$$
 and by \eqref{2.5} and \eqref{2.6}$_1$ the supports $\overline{B_{r_j}(x_j )}$ of the functions $f_j$ are disjoint and contained in $B_{3/4}(0)$ we see by \eqref{2.6}$_2$ that
 \begin{equation}\label{2.12}
  f:=\sum^{\infty}_{j=1}f_j \in C^\infty (\mathbb{R}^n \backslash
  \{0\})\cap L^1 (\mathbb{R}^n )\quad \text{and}\quad \text{supp}(f)\subset B_1(0).
 \end{equation}
 Defining
 $$v_j (y)=\int_{\mathbb{R}^n}\frac{f_j (z)dz}{|y-z|^{n-2}}$$
 and making the change of variables
 $$x=x_j +r_j \xi,\quad y=x_j +r_j \eta,\quad \text{and}\quad z=x_j +r_j \zeta,$$
 we find for $\beta\in[0,n)$ and $R\in[\frac{1}{2},2]$ that
 \begin{align}\label{2.13}
  \notag \int_{|y-x_j |<R}\frac{v_j (y)^\lambda dy}{|x-y|^\beta}&=\int_{|\eta|<R/r_j}\frac{\left(\int_{\mathbb{R}^n}\frac{M_j \psi(\zeta)r^{n}_{j}d\zeta}{r^{n-2}_{j}|\eta-\zeta|^{n-2}}\right)^\lambda}{r^{\beta}_{j}|\xi-\eta|^\beta}r^{n}_{j}d\eta\\
  &=\varepsilon^{\lambda}_{j}\delta^{-\lambda}_{j}r^{n-\beta-(n-2)\lambda}_{j}\int_{|\eta|<R/r_j}\frac{\left(\int_{\mathbb{R}^n}\frac{\psi(\zeta)d\zeta}{|\eta-\zeta|^{n-2}}\right)^\lambda}{|\xi-\eta|^\beta}d\eta.
 \end{align}
 Also
 \begin{equation}\label{2.14}
  0<C_1(n)<\frac{\int_{\mathbb{R}^n}\frac{\psi(\zeta)d\zeta}{|\eta-\zeta|^{n-2}}}{\frac{1}{|\eta|^{n-2}+1}}<C_2 (n)<\infty\quad\text{ for }\eta\in\mathbb{R}^n .
 \end{equation}
 Taking $\beta=0$ and $R=2$ in \eqref{2.13} and using \eqref{2.14} we get
 \begin{equation}\label{2.15}
  \int_{|y-x_j |<2}v_j (y)^\lambda dy\leq C\varepsilon^{\lambda}_{j}\delta^{-\lambda}_{j}r^{n-(n-2)\lambda}_{j}\int_{|\eta|<2/r_j}\left(\frac{1}{|\eta|^{n-2}+1}\right)^\lambda d\eta\leq C(n,\lambda)\varepsilon^{\lambda}_{j}
 \end{equation}
 because $\lambda(n-2)\leq n$.  Defining 
 $$v(x):=\frac{1}{n(n-2)\omega_n}\int_{\mathbb{R}^n}\frac{f(y)dy}{|x-y|^{n-2}}=\frac{1}{n(n-2)\omega_n}\sum^{\infty}_{j=1}v_j
 (x)\quad \text{ for }x\in\mathbb{R}^n$$
 and using \eqref{2.15} we get for $1\leq\lambda\leq\frac{n}{n-2}$ that
 $$n(n-2)\omega_n \| v\|_{L^\lambda (B_1(0))}\leq\sum^{\infty}_{j=1}\| v_j \|_{L^\lambda (B_1(0))}\leq\sum^{\infty}_{j=1}\| v_j \|_{L^\lambda (B_2 (x_j ))}\leq C\sum^{\infty}_{j=1}\varepsilon_j <\infty$$
 by \eqref{2.6}.
 
 If $\lambda\in(0,1)$ then using \eqref{2.15} we see that
 \begin{align*}
  n(n-2)\omega_n \| v\|_{L^\lambda (B_1(0)}&=\|\sum^{\infty}_{j=1}v_j \|_{L^\lambda (B_1(0))}=\left(\int_{B_1(0)}\left(\sum^{\infty}_{j=1}v_j (y)\right)^\lambda dy\right)^{1/\lambda}\\ 
  &\leq\left(\int_{B_1(0)\subset B_2 (x_j )}\sum^{\infty}_{j=1}v_j
    (y)^\lambda dy\right)^{1/\lambda}\\
&\leq\left(\sum^{\infty}_{j=1}\int_{B_2 (x_j )}v_j (y)^\lambda dy\right)^{1/\lambda}\leq\left(\sum^{\infty}_{j=1}C\varepsilon^{\lambda}_{j}\right)^{1/\lambda}<\infty
 \end{align*}
 by \eqref{2.6}.  Thus by \eqref{2.12}
 \begin{equation}\label{2.16}
  v\in C^\infty (\mathbb{R}^n \backslash \{0\})\cap L^\lambda (B_1(0))
 \end{equation}
 and
 \begin{equation}\label{2.17}
  -\Delta v=f\quad\text{ in }\mathbb{R}^n \backslash \{0\}.
 \end{equation}
 Taking $\beta=\alpha\in(0,n)$ and $R=\frac{1}{2}$ in \eqref{2.13} and
 using \eqref{2.14} we find for $|x-x_j |<r_j$ (i.e. $|\xi|<1$) that
 $$(n(n-2)\omega_n )^\lambda \int_{|y|<1}\frac{v(y)^\lambda
   dy}{|x-y|^\alpha}
\ge\int_{|y-x_j|<1/2}\frac{v_j(y)^\lambda
   dy}{|x-y|^\alpha}
\geq C\varepsilon^{\lambda}_{j}\delta^{-\lambda}_{j}r^{n-\alpha-(n-2)\lambda}_{j}I_j (\xi)$$
 where
 \begin{align*}
  I_j (\xi)&:=\int_{|\eta|<\frac{1}{2r_
    j}}\frac{\left(\frac{1}{|\eta|^{n-2}+1}\right)^\lambda}{|\xi-\eta|^\alpha}d\eta\\ 
&\ge\int_{|\eta|<2}\frac{\left(\frac{1}{2^{n-2}+1}\right)^\lambda}{|\xi-\eta|^\alpha}d\eta 
 +\int_{2<|\eta|<\frac{1}{2r_j}}\frac{\left(\frac{1}{2|\eta|^{n-2}}\right)^\lambda}{|\xi-\eta|^\alpha}d\eta\\
  &\geq C(n,\lambda,\alpha)+\left(\frac{2}{3}\right)^\alpha \frac{1}{2^\lambda}\int_{2<|\eta|<\frac{1}{2r_j}}\frac{1}{|\eta|^{(n-2)\lambda+\alpha}}d\eta\\
  &\geq C(n,\lambda,\alpha)
  \begin{cases}
   \frac{1}{r^{n-\alpha-(n-2)\lambda}_{j}} & \text{if }0<\lambda<\frac{n-\alpha}{n-2}\\
   \log\frac{1}{r_j} & \text{if }\lambda=\frac{n-\alpha}{n-2}\\
   1 & \text{if }\frac{n-\alpha}{n-2}<\lambda<\frac{n}{n-2}\\
   1 & \text{if }\lambda=\frac{n}{n-2}.
  \end{cases}
 \end{align*}
 Thus for $|x-x_j |<r_j$ we have
 \begin{equation}\label{2.18}
  \int_{|y|<1}\frac{v(y)^\lambda dy}{|x-y|^\alpha}\geq C(n,\lambda,\alpha)
  \begin{cases}
   \varepsilon^{\lambda}_{j} & \text{if }0<\lambda<\frac{n-\alpha}{n-2}\\
   \varepsilon^{\lambda}_{j}\log\frac{1}{r_j} & \text{if }\lambda=\frac{n-\alpha}{n-2}\\
   \varepsilon^{\lambda}_{j}r^{n-\alpha-(n-2)\lambda}_{j} & \text{if }\frac{n-\alpha}{n-2}<\lambda<\frac{n}{n-2}\\
   \varepsilon^{\lambda}_{j}r^{-\alpha}_{j}(\log\frac{1}{r_j})^{-1} & \text{if }\lambda=\frac{n}{n-2}.
  \end{cases}
 \end{equation}
 Also, for $|x-x_j |<r_j$ we have
 \begin{align}\label{2.19}
  \notag v(x)&\geq\frac{1}{n(n-2)\omega_n}\int_{B_{r_j}(x_j )}\frac{f(y)dy}{|x-y|^{n-2}}\\
  \notag &=C(n)\int_{B_{r_j}(x_j )}\frac{M_j \psi_j (y)}{|x-y|^{n-2}}dy\\
  \notag &=C(n)\int_{|\eta|<1}\frac{M_j \psi(\eta)r^{n}_{j}}{r^{n-2}_{j}|\xi-\eta|^{n-2}}d\eta\\
  &=C(n)\frac{\varepsilon_j}{r^{n-2}_{j}\delta_j}\int_{|\eta|<1}\frac{\psi(\eta)d\eta}{|\xi-\eta|^{n-2}}\geq A\frac{\varepsilon_j}{r^{n-2}_{j}\delta_j}
 \end{align}
 where
 $$A=C(n)\min_{|\xi|\leq1}\int_{|\eta|<1}\frac{\psi(\eta)d\eta}{|\xi-\eta|^{n-2}}>0.$$
 
 Finally, letting $u=\chi v$ where $\chi\in C^\infty (\mathbb{R}^n \to[0,1])$ satisfies $\chi=1$ in $B_2 (0)$ and $\chi=0$ in $\mathbb{R}^n \backslash B_3 (0)$, it follows from \eqref{2.16}--\eqref{2.19} that $u$ satisfies 
 \eqref{2.7}--\eqref{2.11}.
\end{proof}

The following lemma will be needed for the proof of Theorem \ref{thm1.2} when $\lambda>\frac{n}{n-2}$.

\begin{lem}\label{lem2.3}
 Suppose $\alpha\in(0,n)$ and $\lambda>\frac{n}{n-2}$.  Let $\{x_j \}\subset\mathbb{R}^n$, $n\geq3$, and $\{r_j \},\{\varepsilon_j \}\subset(0,1)$ be sequences satisfying
 \begin{equation}\label{2.20}
  0<4|x_{j+1}|<|x_j |<1/2,
 \end{equation}
 \begin{equation}\label{2.21}
  0<r_j <|x_j |/4\quad\text{ and }\quad\sum^{\infty}_{j=1}\varepsilon_j <\infty.
 \end{equation}
 Then there exists a positive function
 \begin{equation}\label{2.22}
  u\in C^\infty (\mathbb{R}^n \backslash\{0\})\cap L^\lambda (\mathbb{R}^n )
 \end{equation}
 such that
 \begin{equation}\label{2.23}
  0\leq-\Delta u\leq\frac{\varepsilon_j}{r^{2+n/\lambda}_{j}}\quad\text{ in }B_{r_j}(x_j )
 \end{equation}
 \begin{equation}\label{2.24}
  -\Delta u=0\quad\text{ in }\mathbb{R}^n \backslash(\{0\}\cup\bigcup^{\infty}_{j=1}B_{r_j}(x_j ))
 \end{equation}
 \begin{equation}\label{2.25}
  u\geq\frac{A\varepsilon_j}{r^{n/\lambda}_{j}}\quad\text{ in }B_{r_j}(x_j )
 \end{equation}
 and
 \begin{equation}\label{2.26}
  \int_{|y-x_j |<r_j}\frac{u(y)^\lambda dy}{|x-y|^\alpha}
\geq\frac{B\varepsilon_j^\lambda}{r^{\alpha}_{j}}\quad\text{ for }x\in B_{r_j}(x_j )
 \end{equation}
 where $A=A(n)$ and $B=B(n,\lambda,\alpha)$ are positive constants.
\end{lem}

\begin{proof}
 Let $\psi:\mathbb{R}^n \to[0,1]$ be a $C^\infty$ function whose support is $\overline{B_1(0)}$.  Define $\psi_j ,f_j :\mathbb{R}^n \to[0,\infty)$ by $\psi_j (y)=\psi(\eta)$ where $y=x_j +r_j \eta$ and $f_j=M_j \psi_j$ where
 $M_j =\frac{\varepsilon_j}{r^{2+n/\lambda}_{j}}$.  Since
 $$\int_{\mathbb{R}^n}f_j (y)dy=M_j \int_{\mathbb{R}^n}\psi(\eta)r^{n}_{j}d\eta=\varepsilon_j r^{\frac{(n-2)\lambda-n}{\lambda}}_{j}\int_{\mathbb{R}^n}\psi(\eta)d\eta\leq\varepsilon_j \int_{\mathbb{R}^n}\psi(\eta)d\eta$$
 and by \eqref{2.20} and \eqref{2.21}$_1$ the supports $\overline{B_{r_j}(x_j )}$ of the functions $f_j$ are disjoint and contained in $B_{3/4}(0)$ we see by \eqref{2.21}$_2$ that
 \begin{equation}\label{2.27}
  f:=\sum^{\infty}_{j=1}f_j \in C^\infty (\mathbb{R}^n \backslash\{0\})\cap L^1 (\mathbb{R}^n )\quad\text{and}\quad \text{supp}(f)\subset B_1(0).
 \end{equation}
 Defining
 $$u_j (y)=\int_{\mathbb{R}^n}\frac{f_j (z)dz}{|y-z|^{n-2}}$$
 and making the change of variables 
 $$x=x_j +r_j \xi, \quad y=x_j +r_j \eta,\quad\text{and}\quad z=x_j +r_j \zeta,$$
 we find for $\beta\in[0,n)$ that
 \begin{align}\label{2.28}
  \notag \int_{\mathbb{R}^n \text{ or }B_{r_j}(x_j )}\frac{u_j (y)^\lambda dy}{|x-y|^\beta}&=\int_{\mathbb{R}^n \text{ or }|\eta|<1}\frac{\left(\int_{\mathbb{R}^n}\frac{M_j \psi(\zeta)r^{n}_{j}d\zeta}{r^{n-2}_{j}|\eta-\zeta|^{n-2}}\right)^\lambda}{r^{\beta}_{j}|\xi-\eta|^\beta}r^{n}_{j}d\eta\\
  &=\varepsilon^{\lambda}_{j}r^{-\beta}_{j}\int_{\mathbb{R}^n \text{ or }|\eta|<1}\frac{\left(\int_{\mathbb{R}^n}\frac{\psi(\zeta)d\zeta}{|\eta-\zeta|^{n-2}}\right)^\lambda}{|\xi-\eta|^\beta}d\eta.
 \end{align}
 Also
 \begin{equation}\label{2.29}
  0<C_1(n)<\frac{\int_{\mathbb{R}^n}\frac{\psi(\zeta) d\zeta}{|\eta-\zeta|^{n-2}}}{\frac{1}{|\eta|^{n-2}+1}}<C_2 (n)<\infty\quad\text{ for }\eta\in\mathbb{R}^n .
 \end{equation}
 Taking $\beta=0$ in \eqref{2.28} and using \eqref{2.29} we get
 \begin{equation}\label{2.30}
  \int_{\mathbb{R}^n}u_j (y)^\lambda dy\leq C(n,\lambda) \varepsilon^{\lambda}_{j}\int_{\mathbb{R}^n}\left(\frac{1}{|\eta|^{n-2}+1}\right)^\lambda d\eta\leq C(n,\lambda)\varepsilon^{\lambda}_{j}
 \end{equation}
 because $\lambda>n/(n-2)$.  Defining
 \begin{equation}\label{2.31}
  u(x):=\frac{1}{n(n-2)\omega_n}\int_{\mathbb{R}^n}\frac{f(y)dy}{|x-y|^{n-2}}=\frac{1}{n(n-2)\omega_n}\sum^{\infty}_{j=1}u_j (x)\quad\text{ for }x\in\mathbb{R}^n
 \end{equation}
 and using \eqref{2.30} we get
 $$n(n-2)\omega_n \| u\|_{L^\lambda (\mathbb{R}^n )}\leq\sum^{\infty}_{j=1}\| u_j \|_{L^\lambda (\mathbb{R}^n )}\leq C\sum^{\infty}_{j=1}\varepsilon_j <\infty$$
 by \eqref{2.21}.  Thus \eqref{2.27} and \eqref{2.31} imply \eqref{2.22} and $-\Delta u=f$ in $\mathbb{R}^n \backslash\{0\}$.  Hence \eqref{2.23} and \eqref{2.24} hold.
 
 Taking $\beta=\alpha\in(0,n)$ in \eqref{2.28} and using \eqref{2.29} we get
 $$\min_{x\in \overline{B_{r_j}(x_j )}}\int_{|y-x_j
   |<r_j}\frac{u(y)^\lambda dy}{|x-y|^\alpha}\geq C(n,\lambda)
 \varepsilon^{\lambda}_{j}r^{-\alpha}_{j}\min_{|\xi|\leq1}\int_{|\eta|<1}\frac{\left(\frac{1}{|\eta|^{n-2}+1}\right)^\lambda}{|\xi-\eta|^\alpha}d\eta
\geq C(n,\lambda,\alpha)\varepsilon^{\lambda}_{j}r^{-\alpha}_{j}$$
 which proves \eqref{2.26}.

 Finally, for $|x-x_j |<r_j$ we have
 \begin{align*}
  u(x)&\geq\frac{1}{n(n-2)\omega_n}\int_{B_{r_j}(x_j )}\frac{f(y)dy}{|x-y|^{n-2}}\\
  &=C(n)\int_{B_{r_j}(x_j )}\frac{M_j \psi_j (y)}{|x-y|^{n-2}}dy\\
  &=C(n)\int_{|\eta|<1}\frac{M_j\psi(\eta)r_j^nd\eta}{r_j^{n-2}|\xi-\eta|^{n-2}}d\eta\\
  &=C(n)\frac{\varepsilon_j}{r^{n/\lambda}_{j}}\int_{|\eta|<1}\frac{\psi(\eta)d\eta}{|\xi-\eta|^{n-2}}\geq\frac{A\varepsilon_j}{r^{n/\lambda}_{j}}
 \end{align*}
 where
 $$A=C(n)\min_{|\xi|\leq1}\int_{|\eta|<1}\frac{\psi(\eta)d\eta}{|\xi-\eta|^{n-2}}>0.$$
 This proves \eqref{2.25}.
 \end{proof}

\begin{lem}\label{lem2.4}
  Suppose for some constants $\alpha\in(0,n)$, $\lambda>0$, and
  $\sigma\ge 0$ that $u$ is a nonnegative solution of
  (\ref{1.1},\ref{1.2}) and $u(x)=O(1)$ as $x\to 0$. Then $u$ has a
  $C^1$ extension to the origin, that is,
  $u=w|_{\mathbb{R}^n\setminus\{0\}}$ for some function $w\in
  C^1(\mathbb{R}^n)$.
\end{lem}

\begin{proof}
Let $v=u+1$. Then by Lemma \ref{lem2.1}, $v$ satisfies
\eqref{2.4}. Since $u$, and hence $v$, is bounded in
$B_1(0)\setminus\{0\}$, the constant $m$ in \eqref{2.4} is zero and by
(\ref{1.1},\ref{1.2}) $-\Delta u$, and hence $-\Delta v$, is bounded
in $B_1(0)\setminus\{0\}$. It therefore follows from \eqref{2.4} that
$v$, and hence $u$, has a $C^1$ extension to the origin.
\end{proof}

\section{The case $0<\lambda<\frac{n-\alpha}{n-2}$}\label{sec3}
In this section we prove Theorem \ref{thm1.1}--\ref{thm1.3} when $0<\lambda<\frac{n-\alpha}{n-2}$.  For these values of $\lambda$, the following theorem implies Theorems \ref{thm1.1} and \ref{thm1.3}.
\begin{thm}\label{thm3.1}
 Suppose $u$ is a nonnegative solution of (\ref{1.1},\ref{1.2}) for some constants $\alpha\in(0,n)$, 
 \begin{equation}\label{3.1}
  0<\lambda<\frac{n-\alpha}{n-2}\quad\text{ and }\quad 0\leq\sigma\leq\frac{n}{n-2}.
 \end{equation}
 Then
 \begin{equation}\label{3.2}
  u(x)=O(|x|^{2-n})\quad\text{ as }x\to 0.
 \end{equation}
\end{thm}

\begin{proof}
 Let $v=u+1$.  Then by Lemma \ref{lem2.1} we have that \eqref{2.1}--\eqref{2.4} hold.  To prove \eqref{3.2}, it clearly suffices to prove
 \begin{equation}\label{3.3}
  v(x)=O(|x|^{2-n})\quad\text{ as }x\to 0.
 \end{equation}
 Choose $\varepsilon\in(0,1)$ such that
 \begin{equation}\label{3.4}
  \lambda<\frac{n-\alpha}{n-2+\varepsilon}.
 \end{equation}
 By \eqref{2.3} we have $v\in L^{\frac{n}{n-2+\varepsilon}}(B_1(0))$ which implies
 $$v^\lambda \in L^{\frac{n}{(n-2+\varepsilon)\lambda}}(B_1(0)).$$
 Thus, since \eqref{3.4} implies
 $$\frac{\lambda(n-2+\varepsilon)}{n}<\frac{n-\alpha}{n},$$
 we have by Riesz potential estimates that
 $$I_{n-\alpha}(v^\lambda )\in L^\infty (B_1(0)).$$
 Hence by \eqref{2.1} and \eqref{2.2}, $v$ is a $C^2$ positive
 solution of 
\[
0\leq-\Delta v\leq Cv^\sigma \quad\text{in } B_1(0)\backslash\{0\}.
\]  
Thus by \eqref{3.1} and \cite[Theorem 2.1]{T2001}, $v$ satisfies \eqref{3.3}.
\end{proof}

Our next result implies Theorem \ref{thm1.2} when $0<\lambda<\frac{n-\alpha}{n-2}$.

\begin{thm}\label{thm3.2}
 Suppose $\alpha,\lambda$, and $\sigma$ are constants satisfying $\alpha\in(0,n)$
 $$0<\lambda<\frac{n-\alpha}{n-2}\quad\text{ and }\quad\sigma>\frac{n}{n-2}.$$
 Let $\varphi:(0,1)\to(0,\infty)$ be a continuous function satisfying
 $$\lim_{t\to 0^+}\varphi(t)=\infty.$$
 Then there exists a nonnegative solution $u$ of (\ref{1.1},\ref{1.2}) such that
 \begin{equation}\label{3.5}
  u(x)\neq O(\varphi(|x|))\quad\text{ as }x\to 0.
 \end{equation}
\end{thm}

\begin{proof}
 Let $\{x_j \}\subset\mathbb{R}^n$ and $\{r_j \},\{\varepsilon_j \}\subset(0,1)$ be sequences satisfying \eqref{2.5} and \eqref{2.6}.  Holding $x_j$ and $\varepsilon_j$ fixed and decreasing $r_j$ to a sufficiently small positive number 
 we can assume
 \begin{equation}\label{3.6}
  \frac{A\varepsilon_j}{r^{n-2}_{j}}>j\varphi(|x_j |)\quad\text{ for }j=1,2,...
 \end{equation}
 and
 \begin{equation}\label{3.7}
  r^{(n-2)\sigma-n}_{j}<A^\sigma B\varepsilon^{\lambda+\sigma-1}_{j}\quad\text{ for }j=1,2,...
 \end{equation}
 where $A$ and $B$ are as in Lemma \ref{lem2.2}.  
 
 Let $u$ be as in Lemma \ref{lem2.2}. By \eqref{2.9}, $u$ satisfies \eqref{1.2} in $B_2 (0)\backslash(\{0\}\cup\cup^{\infty}_{j=1}B_{r_j}(x_j ))$.  Also, for $x\in B_{r_j}(x_j )$, it follows from \eqref{2.8}, \eqref{3.7}, \eqref{2.11}, and
 \eqref{2.10} that
 \begin{align*}
  0&\leq-\Delta u\leq\frac{\varepsilon_j}{r^{n}_{j}}=\frac{r^{(n-2)\sigma-n}_{j}}{A^\sigma B\varepsilon^{\lambda+\sigma-1}_{j}}(B\varepsilon^{\lambda}_{j})\left(\frac{A\varepsilon_j}{r^{n-2}_{j}}\right)^\sigma \\
  &\leq(|x|^{-\alpha}*u^\lambda )u^\sigma.
 \end{align*}
 Thus $u$ satisfies \eqref{1.2} in $B_2 (0)\backslash\{0\}$.  Finally by \eqref{2.10} and \eqref{3.6} we have
 $$u(x_j )\geq\frac{A\varepsilon_j}{r^{n-2}_{j}}>j\varphi(|x_j |)$$
 and thus \eqref{3.5} holds.
\end{proof}

\section{The case $\frac{n-\alpha}{n-2}\leq\lambda<\frac{n}{n-2}$}\label{sec4}
In this section we prove Theorems \ref{thm1.1}--\ref{thm1.3} when
$\frac{n-\alpha}{n-2}\leq\lambda<\frac{n}{n-2}$.  For these values of
$\lambda$, the result below implies Theorem \ref{thm1.1}.

\begin{thm}\label{thm4.1}
 Suppose $u$ is a nonnegative solution of (\ref{1.1},\ref{1.2}) for some constants $\alpha\in(0,n)$, 
 \begin{equation}\label{4.1}
  \frac{n-\alpha}{n-2}\leq\lambda<\frac{n}{n-2}\quad\text{ and }\quad 0\leq\sigma<\frac{2n-\alpha}{n-2}-\lambda.
 \end{equation}
 Then
 \begin{equation}\label{4.2}
  u(x)=O(|x|^{2-n})\quad\text{ as }x\to 0.
 \end{equation}
\end{thm}

\begin{proof}
 Let $v=u+1$.  Then by Lemma \ref{lem2.1} we have that \eqref{2.1}--\eqref{2.4} hold.  To prove \eqref{4.2}, it clearly suffices to prove
 \begin{equation}\label{4.3}
  v(x)=O(|x|^{-(n-2)})\quad\text{ as }x\to 0.
 \end{equation}

 Since increasing $\lambda$ or $\sigma$ increases the right side of the second inequality in \eqref{2.2}$_1$, we can assume instead of \eqref{4.1} that
 \begin{equation}\label{4.4}
  \frac{n-\alpha}{n-2}<\lambda<\frac{n}{n-2},\quad\sigma>0,
\quad\text{ and }\quad 1<\lambda+\sigma<\frac{2n-\alpha}{n-2}.
 \end{equation}
 Since the increased value of $\lambda$ is less than $\frac{n}{n-2}$, it follows from \eqref{2.3} that \eqref{2.1} still holds.
 
 By \eqref{4.4} there exists $\varepsilon=\varepsilon(n,\lambda,\sigma,\alpha)\in(0,1)$ such that
 \begin{equation}\label{4.5}
 \left( \frac{n+2-\alpha}{n+2-\alpha-\varepsilon}\right)\frac{n-\alpha}{n-2}<\lambda<\frac{n}{n-2+\varepsilon}\quad\text{ and }\quad\lambda+\sigma<\frac{2n-\alpha}{n-2+\varepsilon}
 \end{equation}
 which implies
 \begin{equation}\label{4.6}
  \sigma<\frac{2n-\alpha}{n-2+\varepsilon}-\lambda<\frac{2n-\alpha}{n-2+\varepsilon}-\frac{n-\alpha}{n-2}<\frac{n}{n-2+\varepsilon}.
 \end{equation}
 Suppose for contradiction that \eqref{4.3} is false.  Then there is a sequence $\{x_j \}\subset B_{1/2}(0)\backslash \{0\}$ such that $x_j \to 0$ as $j\to\infty$ and
 \begin{equation}\label{4.7}
  \lim_{j\to\infty}|x_j |^{n-2}v(x_j )=\infty.
 \end{equation}
 Since for $|x-x_j |<|x_j |/4$
 $$\int_{\substack{|y-x_j |>|x_j |/2\\
 |y|<1}}
 \frac{-\Delta v(y)}{|x-y|^{n-2}}dy\leq\left(\frac{4}{|x_j |}\right)^{n-2}\int_{|y|<1}-\Delta v(y)dy,$$
 it follows from \eqref{2.3} and \eqref{2.4} that
 \begin{equation}\label{4.8}
  v(x)\leq C\left[\frac{1}{|x_j |^{n-2}}+\int_{|y-x_j |<|x_j |/2}\frac{-\Delta v(y)}{|x-y|^{n-2}}dy\right]\quad\text{ for }|x-x_j |<\frac{|x_j |}{4}.
 \end{equation}
 Substituting $x=x_j$ in \eqref{4.8} and using \eqref{4.7} we find that
 \begin{equation}\label{4.9}
  |x_j |^{n-2}\int_{|y-x_j |<|x_j |/2}\frac{-\Delta v(y)}{|x_j
    -y|^{n-2}}dy
\to\infty\quad\text{ as }j\to\infty.
 \end{equation}
 Also by \eqref{2.3} we have
 \begin{equation}\label{4.10}
  \int_{|y-x_j |<|x_j |/2}-\Delta v(y)dy\to 0\quad\text{ as }j\to\infty.
 \end{equation}
 Defining $f_j (\eta)=-r^{n}_{j}\Delta v(x_j +r_j \eta)$ where $r_j =|x_j |/8$ and making the change of variables $y=x_j +r_j \eta$ in \eqref{4.10} and \eqref{4.9} we get
 \begin{equation}\label{4.11}
  \int_{|\eta|<4}f_j (\eta)d\eta\to 0\quad\text{ as }j\to\infty
 \end{equation}
 and
 \begin{equation}\label{4.12}
  \int_{|\eta|<4}\frac{f_j (\eta)d\eta}{|\eta|^{n-2}}\to\infty\quad\text{ as }j\to\infty.
 \end{equation}
 Let 
 $$N(y)=\int_{|z|<1}\frac{-\Delta v(z)dz}{|y-z|^{n-2}}\quad\text{ for }0<|y|<1.$$
 By \eqref{2.3} and Riesz potential estimates, $N\in L^{\frac{n}{n-2+\varepsilon}}(B_1(0))$.  Thus $N^\lambda \in L^{\frac{n}{\lambda(n-2+\varepsilon)}}(B_1(0))$.  Hence by H\"older's inequality and \eqref{4.5} we have for $R\in(0,1]$ and $|x-x_j |<R|x_j |/8$
 that
 \begin{align}\notag
  \notag \int_{|y|<1}&\frac{N(y)^\lambda dy}{|x-y|^\alpha}-\int_{|y-x_j|<R|x_j|/4}\frac{N(y)^\lambda dy}{|x-y|^\alpha}=\int_{|y-x_j |>R|x_j |/4,|y|<1}\frac{N(y)^\lambda dy}{|x-y|^\alpha}\\  
  \notag &\leq\left(\| N^\lambda \|_{L^{\frac{n}{\lambda(n-2+\varepsilon)}}(B_1(0))}\right)\left(\int_{|y-x_j |>R|x_j |/4}\frac{dy}{|x-y|^{\alpha q}}\right)^{1/q}\text{ where }\frac{\lambda(n-2+\varepsilon)}{n}+\frac{1}{q}=1\\
\label{4.13}  &\leq C\left(\int_{|y-x_j |>R|x_j |/4}\frac{dy}{|y-x_j |^{\alpha q}}\right)^{1/q}\\
\label{4.14}  &=C\frac{1}{|x_j |^{(n-2+\varepsilon)\lambda-(n-\alpha)}}
 \end{align}
 where $C>0$ depends on $R$ but not on $j$.  In \eqref{4.13} we used the fact that
\[
  \frac{|y-x|}{|y-x_j |}\geq\frac{|y-x_j |-|x-x_j |}{|y-x_j|}
=1-\frac{|x-x_j |}{|y-x_j |}>1-\frac{1}{2}=\frac{1}{2}
\]
for $|x-x_j |<R|x_j |/8$ and $|y-x_j |>R|x_j |/4$.
 
 Since, by \eqref{2.3},
 $$N(y)\leq C\left[\frac{1}{|x_j |^{n-2}}+\int_{|z-x_j |<R|x_j |/2}\frac{-\Delta v(z)dz}{|y-z|^{n-2}}\right]\quad\text{ for }|y-x_j |<R|x_j |/4,$$
 we see for $x\in\mathbb{R}^n$ that
 $$\int_{|y-x_j |<R|x_j |/4}\frac{N(y)^\lambda dy}{|x-y|^\alpha}\leq C\left[\frac{1}{|x_j |^{(n-2)\lambda-(n-\alpha)}}+\int_{|y-x_j |<R|x_j |/4}\frac{\left(\int_{z-x_j |<R|x_j |/2}\frac{-\Delta v(z)dz}{|y-z|^{n-2}}\right)^\lambda}{|x-y|^\alpha}dy\right].$$
 It therefore follows from \eqref{2.4}, \eqref{4.4}, \cite[Corollary 3.7]{GT2015},
and \eqref{4.14} that for $|x-x_j |<R|x_j |/8$ we have
 \begin{align*}
  \int_{|y|<1}\frac{v(y)^\lambda dy}{|x-y|^\alpha}&\leq C\left[\int_{|y|<1}\frac{dy}{|x-y|^\alpha |y|^{(n-2)\lambda}}+\int_{|y|<1}\frac{N(y)^\lambda dy}{|x-y|^\alpha}\right]\\
  &\leq C\left[\frac{1}{|x_j |^{(n-2+\varepsilon)\lambda-(n-\alpha)}}+\int_{|y-x_j |<R|x_j |/4}\frac{\left(\int_{|z-x_j |<R|x_j |/2}\frac{-\Delta v(z)dz}{|y-z|^{n-2}}\right)^\lambda}{|x-y|^\alpha}dy\right] 
 \end{align*}
 where $C>0$ depends on $R$ but not on $j$.
 
 We see therefore from \eqref{2.2}, \eqref{2.3}, and \eqref{2.4} 
that for $|x-x_j|<R|x_j|/8$ and $R\in(0,1]$ we have
 \begin{align*}
  -\Delta v(x)\leq &C\left[\frac{1}{|x_j
      |^{(n-2+\varepsilon)\lambda-(n-\alpha)}}+\int_{|y-x_j |<R|x_j
      |/4}\frac{\left(\int_{|z-x_j |<R|x_j |/2}\frac{-\Delta v(z)dz}{|y-z|^{n-2}}\right)^\lambda}{|x-y|^\alpha}dy\right]\\
  &\times\left[\frac{1}{|x_j |^{(n-2)\sigma}}+\left(\int_{|y-x_j |<R|x_j
        |/2}\frac{-\Delta v(y)dy}{|x-y|^{n-2}}\right)^\sigma\right].
 \end{align*}
 Hence under the change of variables
 \[
f_j (\xi)=-r^{n}_{j}\Delta v(x),\quad x=x_j +r_j \xi,
\quad y=x_j +r_j \eta,\quad z=x_j +r_j \zeta,\quad r_j =|x_j |/8
\]
 we obtain from \eqref{4.5} that
 \begin{align}\label{4.15}
  \notag &f_j (\xi)=-r^{n}_{j}\Delta v(x)
\leq-r^{(n-2+\varepsilon)(\lambda+\sigma)-(n-\alpha)}_{j}\Delta v(x)\\
  &\leq C\left[1+\int_{|\eta|<4R}\frac{\left(\int_{|\zeta|<4R}\frac{f_j (\zeta)d\zeta}{|\eta-\zeta|^{n-2}}\right)^\lambda}{|\xi-\eta|^\alpha}d\eta\right]\left[1+\left(\int_{|\eta|<4R}\frac{f_j (\eta)d\eta}{|\xi-\eta|^{n-2}}\right)^\sigma \right]
 \end{align}
 for $|\xi|<R$ where $C>0$ depends on $R$ but not on $j$.
 
To complete the proof of Theorem \ref{thm4.1} we will need the following lemma.
\begin{lem}\label{lem4.1}
 Suppose the sequence
 \begin{equation}\label{4.16}
  \{f_j \}\text{ is bounded in }L^p (B_{4R}(0))
 \end{equation}
 for some constants $p\in[1,\frac{n}{2}]$ and $R\in(0,1]$.  
Then there exists a positive constant $C_0 =C_0 (n,\lambda,\sigma,\alpha)$ such that the sequence 
\begin{equation}\label{4.17}
 \{f_j \}\text{ is bounded in }L^q (B_R (0))
\end{equation}
for some $q\in(p,\infty)$ satisfying
\begin{equation}\label{4.18}
 \frac{1}{p}-\frac{1}{q}\geq C_0.
\end{equation}
\end{lem}

\begin{proof}
 For $R\in(0,1]$ we formally define operators $N_R$ and $I_R$ by
 $$(N_R f)(\xi)=\int_{|\eta|<4R}\frac{f(\eta)d\eta}{|\xi-\eta|^{n-2}}\quad\text{ and }\quad(I_R f)(\xi)=\int_{|\eta|<4R}\frac{f(\eta)d\eta}{|\xi-\eta|^\alpha}.$$
 Define $p_2$ by
 \begin{equation}\label{4.19}
  \frac{1}{p}-\frac{1}{p_2}=\frac{2-\varepsilon}{n}
 \end{equation}
 where $\varepsilon$ is as in \eqref{4.5}.  Then $p_2 \in(p,\infty)$ and thus by Riesz potential estimates we have
 \begin{equation}\label{4.20}
  \|(N_R f_j )^\lambda \|_{p_2 /\lambda}=\| N_R f_j \|^{\lambda}_{p_2}\leq C\| f_j \|^{\lambda}_{p}
 \end{equation}
 and
 \begin{equation}\label{4.21}
  \|(N_R f_j )^\sigma \|_{p_2 /\sigma}=\| N_R f_j \|^{\sigma}_{p_2}\leq C\| f_j \|^{\sigma}_{p}
 \end{equation}
 where $\|\cdot\|_p :=\|\cdot\|_{L^p (B_{4R}(0))}$. Since
 $$\frac{1}{p_2}=\frac{1}{p}-\frac{2-\varepsilon}{n}\leq1-\frac{2-\varepsilon}{n}=\frac{n-2+\varepsilon}{n}$$
 we see by \eqref{4.5} that
 \begin{equation}\label{4.22}
  \frac{p_2}{\lambda}>1.
 \end{equation}
 Now there are two cases to consider.
 \medskip

 \noindent \textbf{Case I.} Suppose
 \begin{equation}\label{4.23}
  \frac{p_2}{\lambda}<\frac{n}{n-\alpha}.
 \end{equation}
 Define $p_3$ and $q$ by
 \begin{equation}\label{4.24}
  \frac{\lambda}{p_2}-\frac{1}{p_3}=\frac{n-\alpha}{n}
 \end{equation}
 and
 \begin{equation}\label{4.25}
  \frac{1}{q}:=\frac{1}{p_3}+\frac{\sigma}{p_2}=\frac{\lambda+\sigma}{p_2}-\frac{n-\alpha}{n}.
 \end{equation}
 It follows from \eqref{4.22}--\eqref{4.25}, \eqref{4.19}, and \eqref{4.4} that
 \begin{equation}\label{4.26}
  1<\frac{p_2}{\lambda}<p_3 <\infty,\quad q>0,
 \end{equation}
 and
 \begin{align*}
  \frac{1}{p}-&\frac{1}{q}=\frac{1}{p}-\left((\lambda+\sigma)\left(\frac{1}{p}-\frac{2-\varepsilon}{n}\right)-\frac{n-\alpha}{n}\right)\\
  &=\frac{(2-\varepsilon)(\lambda+\sigma)+(n-\alpha)}{n}-\frac{\lambda+\sigma-1}{p}\\
  &\geq\frac{(2-\varepsilon)(\lambda+\sigma)+(n-\alpha)-n(\lambda+\sigma-1)}{n}\\
  &=\frac{2n-\alpha-(n-2+\varepsilon)(\lambda+\sigma)}{n}.
 \end{align*}
 Thus \eqref{4.18} holds by \eqref{4.5}.
 
 By \eqref{4.24}, \eqref{4.26}, \eqref{4.20}, and Riesz potential estimates we find that
 \begin{align*}
  \|(I_R ((N_R &f_j )^\lambda ))^q \|_{p_{3/q}}=\| I_R ((N_R f_j )^\lambda )\|^{q}_{p_3}\\
  &\leq C\|(N_R f_j )^\lambda \|^{q}_{p_2 /\lambda}\\
  &\leq C\| f_j \|^{\lambda q}_{p}.
 \end{align*}
 Also by \eqref{4.21} we get
 $$\|(N_R f_j )^{\sigma q}\|_{\frac{p_2}{\sigma q}}=\|(N_R f_j )^\sigma \|^{q}_{p_2 /\sigma}\leq C\| f_j \|^{\sigma q}_{p}.$$
 It therefore follows from \eqref{4.15}, \eqref{4.25}, H\"older's inequality, and \eqref{4.16} that \eqref{4.17} holds.
 \medskip

\noindent \textbf{Case II.} Suppose
\begin{equation}\label{4.27} 
\frac{p_2}{\lambda}\geq\frac{n}{n-\alpha}.
\end{equation} 
Then by Riesz potential estimates, \eqref{4.16}, and \eqref{4.20} we
find that the sequence
 \begin{equation}\label{4.28}
  \{I_R ((N_R f_j )^\lambda)\}\text{ is bounded in }L^{\gamma}(B_{4R}(0))
 \text{ for all } \gamma\in(1,\infty).
\end{equation}
Let $\hat q=p_2/\sigma$. Then by \eqref{4.19},
\[
\frac{1}{p}-\frac{1}{\hat q}=\frac{1}{p}-\frac{\sigma}{p_2}
=\frac{2-\varepsilon}{n}+\frac{1-\sigma}{p_2}.
\] 
Thus for $\sigma\le 1$ we have
\[
\frac{1}{p}-\frac{1}{\hat q}\ge\frac{2-\varepsilon}{n}>0
 \]
and for $\sigma>1$ it follows from \eqref{4.27} and \eqref{4.5} that
\begin{align*}
\frac{1}{p}-\frac{1}{\hat q}
&\ge\frac{2-\varepsilon}{n}-\frac{\sigma-1}{\frac{n\lambda}{n-\alpha}}\\
&\ge \frac{2-\varepsilon}{n}
-\frac
{\frac{2n-\alpha}{n-2}-\lambda-1}
{\frac{n\lambda}{n-\alpha}}\\
&=\frac{n+2-\alpha-\varepsilon}{n\lambda}
\left(\lambda-\left(\frac{n+2-\alpha}{n+2-\alpha-\varepsilon}\right)\frac{n-\alpha}{n-2}\right)>0. 
\end{align*}
Thus defining $q\in(p,\hat q)$ by
\[
\frac{1}{q}=\frac{\frac{1}{p}+\frac{1}{\hat q}}{2}
\]
we have for $\sigma>0$ that
\[
\frac{1}{p}-\frac{1}{q}=\frac{1}{2}\left(\frac{1}{p}-\frac{1}{\hat
    q}\right)
\ge C_0(n,\lambda,\sigma,\alpha)>0.
\]
That is \eqref{4.18} holds.

Since $q\sigma/p_2<\hat q\sigma/p_2=1$ there exists
$\gamma\in(q,\infty)$ such that
\begin{equation}\label{4.29}
\frac{q}{\gamma}+\frac{q\sigma}{p_2}=1.
\end{equation}
Also
\[
\|(I_R ((N_R f_j )^\lambda ))^q \|_{\gamma/q}
=\| I_R ((N_R f_j )^\lambda )\|^{q}_{\gamma}
\]
and by \eqref{4.21}
\[
\|(N_R f_j )^{\sigma q}\|_{\frac{p_2}{\sigma q}}=\|(N_R f_j )^\sigma
\|^{q}_{p_2 /\sigma}\leq C\| f_j \|^{\sigma q}_{p}.
\]
 It therefore follows from \eqref{4.15}, \eqref{4.29}, 
H\"older's inequality, \eqref{4.28}, and \eqref{4.16} that \eqref{4.17} holds.
\end{proof}

We return now to the proof of Theorem \ref{thm4.1}.  By \eqref{4.11} the sequence
\begin{equation}\label{4.30}
 \{f_j \}\text{ is bounded in }L^1(B_4 (0)).
\end{equation}
Starting with this fact and iterating Lemma \ref{lem4.1} a finite number of times ($m$ times is enough if $m>1/C_0$) we see that there exists $R_0 \in(0,1)$ such that the sequence
$$\{f_j \}\text{ is bounded in }L^p (B_{4R_0}(0))$$
for some $p>n/2$.  Hence by Riesz potential estimates the sequence $\{N_{R_0}f_j \}$ is bounded in $L^\infty (B_{4R_0}(0))$.  Thus \eqref{4.15} implies the sequence 
\begin{equation}\label{4.31}
 \{f_j \}\text{ is bounded in }L^\infty (B_{R_0}(0)).
\end{equation}
Since
$$\int_{|\eta|<4}\frac{f_j (\eta)d\eta}{|\eta|^{n-2}}\leq\int_{|\eta|<R_0}\frac{f_j (\eta)d\eta}{|\eta|^{n-2}}+\frac{1}{R^{n-2}_{0}}\int_{R_0 \leq|\eta|<4}f_j (\eta)d\eta$$
we see that \eqref{4.30} and \eqref{4.31} contradict \eqref{4.12}.  This contradiction completes the proof of Theorem \ref{thm4.1}.
\end{proof}

The following theorem implies Theorems \ref{thm1.2} and \ref{thm1.3} when 
$$\frac{n-\alpha}{n-2}\leq\lambda<\frac{n}{n-2}.$$

\begin{thm}\label{thm4.2}
 Suppose $\alpha,\lambda$, and $\sigma$ are constants satisfying
 $$\alpha\in(0,n),\quad\frac{n-\alpha}{n-2}\leq\lambda\leq\frac{n}{n-2}$$
 and
 \begin{align*}
  \sigma\geq\frac{n}{n-2}\qquad&\qquad\text{if }\lambda=\frac{n-\alpha}{n-2};\\
  \sigma>\frac{2n-\alpha}{n-2}-\lambda\,\,&\qquad\text{if } \frac{n-\alpha}{n-2}<\lambda<\frac{n}{n-2};\\
  \sigma>\frac{n-\alpha}{n-2}\qquad&\qquad\text{if }\lambda=\frac{n}{n-2}.
 \end{align*}
 Let $\varphi:(0,1)\to(0,\infty)$ be a continuous function satisfying 
 $$\lim_{t\to 0^+}\varphi(t)=\infty.$$
 Then there exists a nonnegative solution $u$ of (\ref{1.1},\ref{1.2}) such that
 \begin{equation}\label{4.32}
  u(x)\neq O(\varphi(|x|))\quad\text{ as }x\to 0.
 \end{equation}
\end{thm}

\begin{proof}
 Let $\{x_j \}\subset\mathbb{R}^n$ and $\{r_j \},\{\varepsilon_j \}\subset(0,1)$ be sequences satisfying \eqref{2.5} and \eqref{2.6}.  Holding $x_j$ and $\varepsilon_j$ fixed and decreasing $r_j$ to a sufficiently small positive number 
 we can assume for $j=1,2,...$ that
 \begin{equation}\label{4.33}
  j\varphi(|x_j |)\leq 
  \begin{cases}
   \frac{A\varepsilon_j}{r^{n-2}_{j}} & \text{if }\frac{n-\alpha}{n-2}\leq\lambda<\frac{n}{n-2}\\
   \frac{A\varepsilon_j}{r_j^{n-2}\left(\log\frac{1}{r_j}\right)^{\frac{n-2}{n}}} & \text{if }\lambda=\frac{n}{n-2}
  \end{cases}
 \end{equation}
 and
 \begin{equation}\label{4.34}
  A^\sigma B\varepsilon^{\lambda+\sigma-1}_{j}\geq
  \begin{cases}
   r^{(n-2)\sigma-n}_{j}\left(\log\frac{1}{r_j}\right)^{-1} & \text{if }\lambda=\frac{n-\alpha}{n-2}\\
   r^{(n-2)\sigma-(2n-\alpha-(n-2)\lambda)}_{j} & \text{if }\frac{n-\alpha}{n-2}<\lambda<\frac{n}{n-2}\\
   r^{(n-2)\sigma-(n-\alpha)}_{j}\left(\log\frac{1}{r_j}\right)^{\frac{(n-2)\sigma+2}{n}} & \text{if }\lambda=\frac{n}{n-2}
  \end{cases}
 \end{equation}
 where $A$ and $B$ are as in Lemma \ref{lem2.2}.  Let $u$ be as in
 Lemma \ref{lem2.2}.  By \eqref{2.9}, $u$ satisfies \eqref{1.2} in
 $B_2 (0)\backslash(\{0\}\cup\cup^{\infty}_{j=1}B_{r_j}(x_j))$.  Also,
 for $x\in B_{r_j}(x_j )$, it follows from \eqref{2.8}, \eqref{4.34},
 \eqref{2.11}, and \eqref{2.10} that for
 $\frac{n-\alpha}{n-2}\leq\lambda<\frac{n}{n-2}$ we have
 \begin{align*}
  0\leq&-\Delta u\leq\frac{\varepsilon_j}{r^{n}_{j}}\\
  &=
  \begin{cases}
   \frac{r^{(n-2)\sigma-n}_{j}\left(\log\frac{1}{r_j}\right)^{-1}}{A^\sigma B\varepsilon^{\lambda+\sigma-1}_{j}}(B\varepsilon^{\lambda}_{j}\log\frac{1}{r_j})\left(A\frac{\varepsilon_j}{r^{n-2}_{j}}\right)^\sigma&\text{if }\lambda=\frac{n-\alpha}{n-2}\\
   \frac{r^{(n-2)\sigma-(2n-\alpha-(n-2)\lambda)}_{j}}{A^\sigma B\varepsilon^{\lambda+\sigma-1}_{j}}\left(B\varepsilon^{\lambda}_{j}r^{(n-\alpha)-(n-2)\lambda}_{j}\right)\left(A\frac{\varepsilon_j}{r^{n-2}_{j}}\right)^\sigma&\text{if }\frac{n-\alpha}{n-2}<\lambda<\frac{n}{n-2}
  \end{cases}\\
  &\leq(|x|^\alpha *u^\lambda )u^\sigma,
 \end{align*}
 and, for $\lambda=\frac{n}{n-2}$, we have
 \begin{align*}
  0&\leq-\Delta u\leq\frac{\varepsilon_j}{r^{n}_{j}(\log\frac{1}{r_j})^{\frac{n-2}{n}}}\\
  &=\frac{r^{(n-2)\sigma-(n-\alpha)}_{j}(\log\frac{1}{r_j})^{\frac{(n-2)\sigma+2}{n}}}{A^\sigma B\varepsilon^{\lambda+\sigma-1}_{j}}\left(\frac{B\varepsilon^{\lambda}_{j}r^{-\alpha}_{j}}{\log\frac{1}{r_j}}\right)\left(\frac{A\varepsilon_j}{r^{n-2}_{j}(\log\frac{1}{r_j})^{\frac{n-2}{n}}}\right)^\sigma \\
  &\leq(|x|^{-\alpha}*u^\lambda )u^\sigma .
 \end{align*}
 Thus $u$ satisfies \eqref{1.2} in $B_2 (0)\backslash\{0\}$.  Finally, by \eqref{2.10} and \eqref{4.33} we have
 $$u(x_j )\geq j\varphi(|x_j|)$$
 and thus \eqref{4.32} holds.
\end{proof}

\section{The case $\lambda\geq\frac{n}{n-2}$}\label{sec5}
In this section we prove Theorems \ref{thm1.1}--\ref{thm1.3} when
$\lambda\geq\frac{n}{n-2}$. For these values of $\lambda$, our
next result implies Theorem \ref{thm1.1}.
\begin{thm}\label{thm5.1}
 Suppose $u$ is a nonnegative solution of (\ref{1.1},\ref{1.2}) for some constants $\alpha\in(0,n)$,
 \begin{equation}\label{5.1}
  \lambda\geq\frac{n}{n-2}\quad \text{ and }\quad 0\leq\sigma<1-\frac{\alpha-2}{n}\lambda.
 \end{equation}
 Then
 \begin{equation}\label{5.2}
  u(x)=O(1)\quad\text{ as }x\to 0
 \end{equation}
and $u$ has a $C^1$ extension to the origin.
\end{thm}
 
\begin{proof}
 Let $v=u+1$.  Then by Lemma \ref{lem2.1} we have that \eqref{2.1}--\eqref{2.4} hold.  To prove \eqref{5.2} it clearly suffices to prove
 \begin{equation}\label{5.3}
  v(x)=O(1)\quad\text{ as }x\to 0.
 \end{equation}

 By \eqref{2.1} and \eqref{5.1}, the constant $m$ in \eqref{2.4} is zero and thus by \eqref{2.4}
 \begin{equation}\label{5.4}
  v(x)\leq C\left[1+\int_{|y|<1}\frac{-\Delta v(y)}{|x-y|^{n-2}}dy\right]\quad\text{ for }0<|x|<1
 \end{equation}
 for some positive constant $C$.
 
 Since increasing $\sigma$ increases the right side of the second inequality in \eqref{2.2}$_1$, we can assume instead of \eqref{5.1} that
 \begin{equation}\label{5.5}
  \lambda\geq\frac{n}{n-2}\quad\text{ and }\quad 0<\sigma<1-\frac{\alpha-2}{n}\lambda
 \end{equation}
 which implies
 \begin{equation}\label{5.6}
  \frac{\sigma}{\lambda}<\frac{2-\alpha}{n}+\frac{1}{\lambda}\leq\frac{2-\alpha}{n}+\frac{n-2}{n}=\frac{n-\alpha}{n}.
 \end{equation}
 By \eqref{5.5} there exists
 $\varepsilon=\varepsilon(n,\lambda,\sigma,\alpha)\in(0,1)$ such that
 \[
\alpha+\varepsilon<n\quad\text{ and
}\quad\sigma<1-\frac{\alpha+\varepsilon-2}{n}\lambda
\]
 which implies
 \begin{equation}\label{5.7}
  \frac{\sigma-1}{\lambda}<\frac{2-\alpha-\varepsilon}{n}.
 \end{equation}
 
 For the proof of Theorem \ref{thm5.1} we will need the following lemma.
 \begin{lem}\label{lem5.1}
  Suppose
  \begin{equation}\label{5.8}
   v\in L^p (B_1(0))
  \end{equation}
  for some constant
  \begin{equation}\label{5.9}
   p\in\left[\lambda,\frac{n\lambda}{n-\alpha-\varepsilon}\right).
  \end{equation}
  Then either
  \begin{equation}\label{5.10}
   v\in L^{\frac{n\lambda}{n-\alpha-\varepsilon}}(B_1(0))
  \end{equation}
  or there exists a positive constant $C_0 =C_0 (n,\lambda,\sigma,\alpha)$ such that
  \begin{equation}\label{5.11}
   v\in L^q (B_1(0))
  \end{equation}
  for some $q\in(p,\infty)$ satisfying
  \begin{equation}\label{5.12}
   \frac{1}{p}-\frac{1}{q}\geq C_0.
  \end{equation}
 \end{lem} 

 \begin{proof}
  Define $p_2$ by
  \begin{equation}\label{5.13}
   \frac{\lambda}{p}-\frac{1}{p_2}=\frac{n-\alpha-\varepsilon}{n}.
  \end{equation}
  Then by \eqref{5.9}
  $$1\leq\frac{p}{\lambda}<p_2 <\infty$$
  and thus by Riesz potential estimates and \eqref{5.8} we have
  \begin{equation}\label{5.14}
   \| I_{n-\alpha}v^\lambda \|_{p_2}\leq C\| v^\lambda \|_{\frac{p}{\lambda}}=C\| v\|^{\lambda}_{p}<\infty
  \end{equation}
  where $I_\beta$ is defined in \eqref{2.2.5}.

Define $p_3 >0$ by
  \begin{equation}\label{5.15}
   \frac{1}{p_3}=\frac{1}{p_2}+\frac{\sigma}{p}.
  \end{equation}
  Then by H\"older's inequality
  \begin{align*}
   \|((I_{n-\alpha}v^\lambda )&v^\sigma )^{p_3}\|_1
   \leq\|(I_{n-\alpha}v^\lambda
   )^{p_3}\|_{\frac{p_2}{p_3}}\| v^{\sigma
     p_3}\|_{\frac{p}{\sigma p_3}}\\
   &=\| I_{n-\alpha}v^\lambda \|^{p_3}_{p_2}\|
   v\|^{\sigma p_3}_{p}<\infty
  \end{align*}
  by \eqref{5.8} and \eqref{5.14}.  Hence by \eqref{2.2}
  \begin{equation}\label{5.16}
   -\Delta v\in L^{p_3}(B_1(0)).
  \end{equation}
  Also by \eqref{5.15}, \eqref{5.13}, \eqref{5.9}, and \eqref{5.7} we have
  \begin{align*}
   \frac{1}{p_3}&=\frac{\lambda+\sigma}{p}-\frac{n-\alpha-\varepsilon}{n}\leq\frac{\lambda+\sigma}{\lambda}-\frac{n-\alpha-\varepsilon}{n}\\
   &=\frac{\sigma}{\lambda}+\frac{\alpha+\varepsilon}{n}<\frac{1}{\lambda}-\frac{\alpha+\varepsilon-2}{n}+\frac{\alpha+\varepsilon}{n}\\
   &=\frac{1}{\lambda}+\frac{2}{n}.
  \end{align*}
  Thus by \eqref{5.5} we see that
  \begin{equation}\label{5.17}
   p_3 >1.
  \end{equation}
  
  \noindent \textbf{Case I.} Suppose $p_3 \geq\frac{n}{2}$.  Then by \eqref{5.16}, \eqref{5.4}, and Riesz potential estimates we have $v\in L^q (B_1(0))$ for all $q>1$ which implies \eqref{5.10}.
  \medskip

  \noindent \textbf{Case II.}  Suppose $p_3 <\frac{n}{2}$.  Define $q$ by
  \begin{equation}\label{5.18}
   \frac{1}{p_3}-\frac{1}{q}=\frac{2}{n}.
  \end{equation}
  Then by \eqref{5.17}
  $$1<p_3 <q<\infty.$$
  Hence by \eqref{5.16}, \eqref{5.4} and Riesz potential estimates we have \eqref{5.11} holds.  
  
  Also by \eqref{5.18}, \eqref{5.15}, \eqref{5.13}. \eqref{5.9}, and \eqref{5.7} we get
  \begin{align*}
   \frac{1}{p}-\frac{1}{q}&=\frac{1}{p}+\frac{2}{n}-\frac{1}{p_3}=\frac{1}{p}+\frac{2}{n}-\frac{\sigma}{p}-\frac{\lambda}{p}+1-\frac{\alpha+\varepsilon}{n}\\
   &=-\frac{\lambda+\sigma-1}{p}+1-\frac{\alpha+\varepsilon-2}{n}\\
   &\geq\frac{1-(\lambda+\sigma)}{\lambda}+1+\frac{2-\alpha-\varepsilon}{n}>0.
  \end{align*}
  Thus \eqref{5.12} holds.
 \end{proof} 
  
  We now return to the proof of Theorem \ref{thm5.1}.  By \eqref{2.1}, $v\in L^\lambda (B_1(0))$.  Starting with this fact and iterating Lemma \ref{lem5.1} a finite number of times we see that \eqref{5.10} holds.  In particular
  \begin{equation}\label{5.19}
   v\in L^p (B_1(0))
  \end{equation}
  for some
  \begin{equation}\label{5.20}
   p>\frac{n\lambda}{n-\alpha}.
  \end{equation}
  Hence $v^\lambda \in L^{\frac{p}{\lambda}}(B_1(0))$ and
  $\frac{p}{\lambda}>\frac{n}{n-\alpha}$.  Thus by Riesz potential
  estimates $I_{n-\alpha}(v^\lambda )\in L^\infty (B_1(0))$.  So by
  \eqref{2.2}
  \begin{equation}\label{5.21}
   0\leq-\Delta v<Cv^\sigma \text{ in }B_1(0)\backslash\{0\}.
  \end{equation}
  Hence by \eqref{5.19}, $-\Delta v\in L^{\frac{p}{\sigma}}(B_1(0))$ and by \eqref{5.20} and \eqref{5.6}
  $$\frac{p}{\sigma}>\frac{n}{n-\alpha}\frac{\lambda}{\sigma}>\left(\frac{n}{n-\alpha}\right)^2 >1.$$
  Thus by \eqref{5.4} and Riesz potential estimates
  \begin{equation}\label{5.22}
   v\in L^q (B_1(0))\quad\text{ where }q=
   \begin{cases}
    \infty & \text{if }\frac{p}{\sigma}\geq\frac{n}{2-\varepsilon}\\
    \frac{1}{\frac{\sigma}{p}-\frac{2-\varepsilon}{n}} & \text{if }\frac{p}{\sigma}<\frac{n}{2-\varepsilon}.
   \end{cases}
  \end{equation}
  If $q=\infty$ then \eqref{5.3} holds.  Hence we can assume $\frac{p}{\sigma}<\frac{n}{2-\varepsilon}$.  Then by \eqref{5.22}
  $$\frac{1}{p}-\frac{1}{q}=\frac{1-\sigma}{p}+\frac{2-\varepsilon}{n}.$$
  Thus, if $\sigma\in(0,1]$ then 
  $$\frac{1}{p}-\frac{1}{q}>\frac{1}{n}.$$
  On the other hand, if $\sigma>1$ then by \eqref{5.20} and \eqref{5.7}
  \begin{align*}
   \frac{1}{p}-\frac{1}{q}&=\frac{2-\varepsilon}{n}-\frac{\sigma-1}{p}\\
   &>\frac{2-\varepsilon}{n}-\frac{\sigma-1}{\lambda}\frac{n-\alpha}{n}\\
   &>\frac{2-\varepsilon}{n}-\frac{2-\alpha-\varepsilon}{n}=\frac{\alpha}{n}.
  \end{align*}
  Thus for $\sigma>0$ we have
  $$\frac{1}{p}-\frac{1}{q}>C(n,\alpha)>0.$$
  Hence, after a finite number of iterations of the procedure of going from \eqref{5.19} to \eqref{5.22} we get $v\in L^\infty (B_1(0))$ and hence we see again that \eqref{5.3} holds.

Finally by Lemma \ref{lem2.4}, $u$ has a $C^1$ extension to the origin.
\end{proof}

The result below implies Theorems \ref{thm1.2} and \ref{thm1.3}
when $\lambda\geq\frac{n}{n-2}$.
\begin{thm}\label{thm5.2}
  Suppose $\alpha,\lambda$, and $\sigma$ are constants satisfying
  $\alpha\in(0,n)$,
\[
\lambda\geq\frac{n}{n-2},\quad \sigma\ge 0,\quad\text{ and
}\quad\sigma>1-\frac{\alpha-2}{n}\lambda.
\]
 Let $\varphi:(0,1)\to(0,\infty)$ be a continuous function satisfying
 $$\lim_{t\to 0^+}\varphi(t)=\infty.$$
 Then there exists a nonnegative solution $u$ of (\ref{1.1},\ref{1.2}) such that
 \begin{equation}\label{5.23}
  u(x)\neq O(\varphi(|x|))\quad\text{ as }x\to 0.
 \end{equation}
\end{thm}

\begin{proof}
 If $\lambda=\frac{n}{n-2}$ then Theorem \ref{thm5.2} follows 
from Theorem \ref{thm4.2}.  Hence we can assume $\lambda>\frac{n}{n-2}$.
 
 Let $\{x_j \}\subset\mathbb{R}^n$ and $\{r_j \},\{e_j \}\subset(0,1)$ be sequences satisfying \eqref{2.20} and \eqref{2.21}.  Holding $x_j$ and $\varepsilon_j$ fixed and decreasing $r_j$ to a sufficiently small positive number we 
 can assume
 \begin{equation}\label{5.24}
  \frac{A\varepsilon_j}{r^{n/\lambda}_{j}}>j\varphi(|x_j |)\quad\text{ for }j=1,2,...
 \end{equation}
 and
 \begin{equation}\label{5.25}
  r^{\frac{n}{\lambda}(\sigma-(1-\frac{\alpha-2}{n}\lambda))}_{j}<A^\sigma B\varepsilon^{\lambda+\sigma-1}_{j}\quad\text{ for }j=1,2,...
 \end{equation}
 where $A$ and $B$ are as in Lemma \ref{lem2.3}.  
Let $u$ be as in Lemma \ref{lem2.3}.  By \eqref{2.24} $u$ satisfies \eqref{1.2} in $B_2 (0)\backslash(\{0\}\cup\cup^{\infty}_{j=1}B_{r_j}(x_j ))$.  Also, for $x\in B_{r_j}(x_j )$, it follows from 
 \eqref{2.23}, \eqref{5.25}, \eqref{2.26}, and \eqref{2.25} that 
 \begin{align*}
  0&\leq-\Delta u\leq\frac{\varepsilon_j}{r^{2+n/\lambda}_{j}}=\frac{r^{\frac{n}{\lambda}(\sigma-(1-\frac{\alpha-2}{n}\lambda))}_{j}}{A^\sigma B\varepsilon^{\lambda+\sigma-1}_{j}}\left(\frac{B\varepsilon^{\lambda}_{j}}{r^\alpha_{j}}\right)\left(\frac{A\varepsilon_j}{r^{n/\lambda}_{j}}\right)^\sigma \\
  &\leq(|x|^\alpha *u^\lambda )u^\sigma .
 \end{align*}
 Thus $u$ satisfies \eqref{1.2} in $B_2 (0)\backslash\{0\}$.  
 
 Finally by \eqref{2.25} and \eqref{5.24} we have
 $$u(x_j )\geq\frac{A\varepsilon_j}{r^{n/\lambda}_{j}}>j\varphi(|x_j |)$$
 and thus \eqref{5.23} holds.
\end{proof}

\end{document}